\documentclass{article}

\usepackage{amsfonts} % Includes
\usepackage{amsmath} % all the
\usepackage{amssymb} % AMS
\usepackage{amsthm} % stuff!
	\newtheorem{thm}{Theorem}[section]

\usepackage{authblk}

\usepackage[margin=3cm]{geometry} % Geometry package good for setting margins
\usepackage{hyperref} % hyperlinks in document when building pdf
\usepackage{makecell} % for split table cells
\usepackage{multicol}

\usepackage{graphicx}

\newcommand{\To}{\longrightarrow}

\newcommand{\R}{\mathbb{R}}

\DeclareMathOperator{\Cl}{Cl}

\begin{document}

\title{The sensitivity conjecture, induced subgraphs of cubes, and Clifford algebras} 

\author{Daniel V. Mathews}
\affil{School of Mathematics,
Monash University}

\author{Daniel V. Mathews}

\date{}

\maketitle

\begin{abstract}
We give another version of Huang's proof that an induced subgraph of the $n$-dimensional cube graph containing over half the vertices has maximal degree at least $\sqrt{n}$, which implies the Sensitivity Conjecture.
This argument uses Clifford algebras of positive definite signature in a natural way. We also prove a weighted version of the result.
\end{abstract}

%\tableofcontents

\section{Introduction}

In \cite{Huang_Induced_subgraphs}, Huang proves a theorem about induced subgraphs of hypercube graphs, and uses it to prove the Sensitivity Conjecture.
The main result of that paper is as follows. Let $n \geq 1$ be a integer, and let $Q^n$ be the $n$-dimensional (hyper)cube graph, with $2^n$ vertices given by $\{0,1\}^n$ and edges connecting points which differ at exactly one coordinate. For a graph $G$, let $\Delta(G)$ denote the maximum degree of its vertices.

\begin{thm}
\label{thm:main_thm}
Let $H$ be a subgraph of $Q^n$ induced by a subset of $(2^{n-1} + 1)$ vertices. Then $\Delta(H) \geq \sqrt{n}$. 
\end{thm}

Huang's proof uses a sequence of matrices $A_n$, which have size $2^n \times 2^n$. In this short paper we observe that these matrices are the matrices of multiplication by a certain natural element $S$ in positive definite Clifford algebras. From this we are able to reformulate the proof in terms of Clifford algebras. 

We note that Karasev in \cite{Karasev_exterior_algebra} has related Huang's matrices $A_n$ to exterior algebras; these are special cases of Clifford algebras when the quadratic form is zero. By using a nontrivial quadratic form for our Clifford algebras we are able to see a close connection between the multiplicative structure of the Clifford algebras, and the combinatorics of the cube graph and its subgraphs.

Our approach generalises to give a ``weighted" version of the theorem as follows. Let $a_1, \ldots, a_n$ be non-negative real numbers. Each edge of $Q^n$ joins points whose coordinates differ in one place; for a vertex $v$ denote by $v(i)$ the unique vertex which differs from $v$ only in the $i$th place, so that the vertices adjacent to $v$ in $Q^n$ are precisely $v(1), \ldots, v(n)$. Let edges joining points whose coordinates differ in the $i$th place (i.e. vertices of the form $v$ and $v(i)$) have weight $a_i$. Then for a vertex $v$ of a subgraph $H$ of $Q^n$, its \emph{weighted degree} is the sum of the weights on adjacent edges. Denote by $\Delta_a (H)$ the maximum weighted degree of the vertices of $H$.
\begin{thm}
\label{thm:weighted_thm}
Let $H$ be a subgraph of $Q^n$ induced by a subset of $(2^{n-1} + 1)$ vertices, and let the weights $a_1, \ldots, a_n$ be any non-negative real numbers. Then $\Delta_a (H) \geq \sqrt{a_1^2 + \cdots + a_n^2}$.
\end{thm}
Setting $a_1 = \cdots = a_n = 1$ in theorem \ref{thm:weighted_thm} recovers theorem \ref{thm:main_thm}.

A connection between Huang's proof and Clifford algebras has also been given by Tao \cite{Tao_Huang_blog}.

After the first version of this paper was posted, the author was informed that essentially identical observations (including theorems \ref{thm:main_thm} and \ref{thm:weighted_thm}) had independently been made by T. Mrowka some weeks earlier, although they had not been been published \cite{Mrowka_email}.

{\flushleft \textbf{Acknowledgments.}} The author is supported by Australian Research Council grant DP160103085.

\section{Clifford algebras}

\subsection{Background}

We recall some well-known notions about Clifford algebras, which can be found in any standard text or introductory article on the subject (e.g. \cite{Harvey_Spinors_and_calibrations, Lounesto_Clifford, Todorov_Clifford}).
Let $V$ be a real vector space, equipped with a symmetric bilinear form $B \colon V \times V \To \R$, or equivalently, a quadratic form $Q \colon V \To \R$, related by $Q(v) = B(v,v)$ and $B(v,w) = \frac{1}{2} \left( Q(v+w) - Q(v) - Q(w) \right)$. The \emph{Clifford algebra} $\Cl(V,Q)$ is the associative algebra freely generated by $V$, subject to the relations
\[
v^2 = B(v,v) = Q(v)
\quad \text{or equivalently} \quad
vw+wv = 2B(v,w)
\]
for all $v,w \in V$. Alternatively, $\Cl(V,Q)$ is given by the tensor algebra of $V$, modulo the ideal generated by elements of the form $v \otimes v - Q(v)$.

We are interested here in positive definite $B$ and $Q$. If $V$ has dimension $n$, and $e_1, \ldots, e_n$ form an orthonormal basis of $V$ with respect to $B$, then the relations imply that for all $1 \leq i,j \leq n$,
\[
e_i^2 = 1
\quad \text{and} \quad
e_i e_j = - e_j e_i.
\]
As a vector space, $\Cl(V,Q)$ has dimension $2^n$ and a basis is given by $e_{i_1} \cdots e_{i_k}$, for each sequence $1 \leq i_1 < \cdots < i_k \leq n$. This includes the empty sequence, whose corresponding basis element is the identity $1$. In the positive definite case $\Cl(V,Q)$ only depends on the dimension $n$ and we write $\Cl(n)$ for $\Cl(V,Q)$. It is known that $\Cl(n)$ is given as a direct sum of one or two matrix algebras, over the real numbers, complex numbers or quaternions, depending on $n$ modulo $8$.

The involution $\phi \colon V \To V$ given by $\phi(v) = -v$ extends to an involution of $\Cl(n)$, which we also denote $\phi$, such that for each $1 \leq i_1 < \cdots < i_k \leq n$, $\phi(e_{i_1} \cdots e_{i_k}) = (-1)^k e_{i_1} \cdots e_{i_k}$.

\subsection{Clifford algebras and cubes}

We can identify the basis element $e_{i_1} \cdots e_{i_k}$ (where $1 \leq i_1 < \cdots < i_k \leq n$) of for $\Cl(n)$ with the vertex of $Q^n$ with ones in positions $i_1, \ldots, i_k$ and zeroes elsewhere. Thus we write $e_v$ for the basis element corresponding to the vertex $v \in \{0,1\}^n$. A general element $x \in \Cl(n)$ can be written uniquely as $x = \sum_v x_v e_v$, where the sum is over vertices $v \in \{0,1\}^n$ of $Q^n$ and each $x_v \in \R$.

Observe that multiplying a basis element $e_v$ by $e_i$, for $1 \leq i \leq n$ sends $e_v$ to $\pm e_{v(i)}$. In other words, multiplication by $e_i$, up to sign, permutes basis elements by translating them along edges of $Q^n$ in the $i$th direction.

\subsection{A Clifford element for counting degrees}

Consider the element $S = e_1 + \cdots + e_n \in \Cl(n)$. Observe that $S^2 = n$.

For each vertex $v$ then $e_v S$ is a signed sum of $e_w$ over the vertices $w$ adjacent to $v$ in $Q^n$, as is $S e_V$ (but in general with different signs). That is, $e_v S = \sum_{i=1}^n \varepsilon_i e_{v(i)}$ where each $\varepsilon_i = \pm 1$.

Similarly, if $x \in \Cl(n)$ is a sum of basis elements $x = \sum_{w\in W} e_w$, over some subset $W \subseteq \{0,1\}^n$, then the coefficient of $e_v$ in $x S$ is given by a sum of $\pm 1$s, one for each vertex of $W$ adjacent to $v$. Thus the coefficient of $e_v$ in $xS$ is bounded above by the degree of $v$ in the subgraph of $Q^n$ induced by $W$. In this way, multiplication by $S$ can be used to bound degrees of vertices.

We therefore consider the map $M_S \colon \Cl(n) \To \Cl(n)$ given by multiplication on the right by $S$, i.e. $M_S (v) = v S$. Observe that
\[
\left( \sqrt{n} + S \right) S 
%= \sqrt{n} S + n 
= \sqrt{n} \left( \sqrt{n} + S \right)
\quad \text{and} \quad
\left( -\sqrt{n} + S \right) S
%= -\sqrt{n} S + n
= -\sqrt{n} \left( - \sqrt{n} + S \right)
\]
so $\sqrt{n} + S$ and $-\sqrt{n} + S$ are eigenvectors of $M_S$ with eigenvalues $\sqrt{n}$ and $-\sqrt{n}$ respectively. For convenience write $\alpha_+ = \sqrt{n} + S$ and $\alpha_- = - \sqrt{n} + S$. 

Indeed then the principal left ideals $\Cl(n) \alpha_+$ and $\Cl(n) \alpha_-$ are contained in the $\sqrt{n}$ and $-\sqrt{n}$ eigenspaces respectively. Hence their intersection is zero, but as $\alpha_+ - \alpha_- = 2\sqrt{n}$, they span $\Cl(n)$. Hence they are the entire eigenspaces and we have
\[
\Cl(n) = \Cl(n) \alpha_+ \oplus \Cl(n) \alpha_- = \ker \left( M_S - \sqrt{n} \right) \oplus \ker \left( M_S + \sqrt{n} \right).
\]
Since $\alpha_+ \alpha_- = \alpha_- \alpha_+ = 0$, this is in fact a product of rings.
Indeed, we have $\alpha_+^2 = 2 \sqrt{n} \alpha_+$ and $\alpha_-^2 = -2\sqrt{n} \alpha_-$, so $\frac{1}{2\sqrt{n}} \alpha_+$ and $\frac{1}{2\sqrt{n}} \alpha_-$ are complementary orthogonal idempotents. One can check that the matrix of $M_S$ with respect to the lexicographically ordered basis is the matrix $A_n$ of \cite{Huang_Induced_subgraphs}.

The involution $\phi$ takes $\Cl(n) \alpha_+$ to $\Cl(n) \alpha_-$ and vice versa, hence both have dimension $2^{n-1}$ as vector spaces.

More generally, if we take real numbers $a_1, \ldots, a_n$ we can define $S_a = a_1 e_1 + \cdots + a_n e_n$. If $x = \sum_{w \in W} e_w$, over some subset $W \subseteq \{0,1\}^n$, then the coefficient of $e_v$ in $xS$ is given by a sum of $\pm a_i$ terms, one for each vertex of $W$ adjacent to $v$. That is, the coefficient of $e_v$ in $xS$ is $\sum_i \varepsilon_i a_i$, where each $\varepsilon_i = \pm 1$, and the sum is over $1 \leq i \leq n$ such that $v(i) \in W$. This coefficient is bounded above by the weighted degree of $v$ in the subgraph of $Q^n$ induced by $w$. We have $S_a^2 = a_1^2 + \cdots + a_n^2$; multiplication $M_{S_a}$ by $S_a$ on the right has eigenvectors $\alpha_\pm = \pm \sqrt{a_1^2 + \cdots + a_n^2} + S_a$ with eigenvalues $\pm \sqrt{a_1^2 + \cdots + a_n^2}$. If not all $a_i$ are zero, we obtain a direct sum of rings $\Cl(n) \alpha_{+} \oplus \Cl(n) \alpha_{-} = \ker \left( M_{S_a} - \sqrt{a_1^2 + \cdots + a_n^2} \right) \oplus \ker \left( M_{S_a} + \sqrt{a_1^2 + \cdots + a_n^2} \right)$.

\section{Proofs of theorems}

Let $W \subset \{0,1\}^n$ be a set of $2^{n-1}+1$ vertices of $Q^n$, and let $H$ be the subgraph of $Q^n$ induced by $W$. For a vertex $v \in W$, denote by $N(v)$ its neighbours in $H$, i.e. those vertices in $W$ adjacent to $v$. Thus $|N(v)|$ is the degree of $v$ in $W$, and in theorem \ref{thm:main_thm} we want to show some $|N(v)| \geq \sqrt{n}$.

Let $C_H \subset \Cl(n)$ be the vector subspace spanned by $e_w$, over all $w \in W$. In other words, $C_H = \bigoplus_{w \in W} \R e_w$.

\begin{proof}[Proof of theorem \ref{thm:main_thm}]
As $C_H$ has dimension $2^{n-1} + 1$, it has nontrivial intersection with $\Cl(n) \alpha_+ = \ker(M_S - \sqrt{n})$. Let $0 \neq x = \sum_{w \in W} x_w e_w$ lie in the intersection. Then $x(S - \sqrt{n}) = 0$, so $xS = \sqrt{n} x$.

Now in $xS$, for any $v \in \{0,1\}^n$, the coefficient of $e_v$ is given by a sum $\sum_{w \in N(v)} \varepsilon_w x_w$, where each $\varepsilon_w = \pm 1$, which is at most $\sum_{w \in N(v)} \left| x_w \right|$ in absolute value. On the other hand, in $\sqrt{n} \; x$ the coefficient of $e_v$ is of course $\sqrt{n} \; x_v$. So $\sqrt{n} \; x_v = \sum_{w \in N(v)} \varepsilon_w x_w$ and hence $\sqrt{n} \left| x_v \right| \leq \sum_{w \in N(v)} \left| x_w \right|$.

This implies that some vertex of $H$ has degree at least $\sqrt{n}$. Indeed, let $v_0$ be a vertex such that the coefficient $x_v$ is largest in absolute value, i.e. $|x_{v_0}| \geq |x_v|$ for all $v \in \{0,1\}^n$. Then we have
\[
\sqrt{n} \; \left| x_{v_0} \right| 
%= \left| \sum_{w \in N(v_0)} \varepsilon_w x_w \right|
%\leq \sum_{w \in N(v_0)} \left| \varepsilon_w x_w \right|
%= 
\leq \sum_{w \in N(v_0)} \left| x_w \right|
\leq \sum_{w \in N(v_0)} \left| x_{v_0} \right|
= \left| N(v_0) \right| \; \left| x_{v_0} \right|.
\]
Thus $\left| N(v_0) \right| \geq \sqrt{n}$, i.e. $v_0$ has degree at least $\sqrt{n}$, so $\Delta(H) \geq \sqrt{n}$.
\end{proof}

\begin{proof}[Proof of theorem \ref{thm:weighted_thm}]
If all $a_i$ are zero the result is immediate, so assume $a_1^2 + \cdots + a_n^2 > 0$.
We again have a nontrivial intersection of $C_H$ with $\Cl(n) \alpha_{+} = \ker \left( M_{S_a} - \sqrt{a_1^2 + \cdots + a_n^2} \right)$, and we take $0 \neq x = \sum_{w \in W} x_w e_w$ in the intersection. We have 
%$x(S_a - \sqrt{a_1^2 + \cdots + a_n^2} ) = 0$, so 
$x S_a = \sqrt{a_1^2 + \cdots + a_n^2} \; x$, and comparing coefficients of $e_v$ gives $\sqrt{a_1^2 + \cdots + a_n^2} \; x_v = \sum_{i=1}^n \varepsilon_i a_i x_{v(i)}$; here each $\varepsilon_i = \pm 1$ and the nonzero terms in this sum correspond precisely to the neighbours of $v$ in the vertex set $W$ of $H$. Taking $v_0$ such that $x_{v_0}$ is largest in absolute value we then have
\[
\sqrt{a_1^2 + \cdots + a_n^2} \; \left| x_{v_0} \right|
= \left| \sum_{i=1}^n \varepsilon_i a_i x_{v_0 (i)} \right|
\leq \sum_{i=1}^n a_i \left| x_{v_0 (i)} \right| 
\leq \left| x_{v_0} \right| \sum_{v_0 (i) \in W} a_i 
\]
Now $\sum_{v_0(i) \in W} a_i$ is the weighted degree of $v_0$ in $W$, and we have shown it is at least $\sqrt{a_1^2 + \cdots + a_n^2}$.
\end{proof}

\small

\bibliography{sensitivity_conjecture_clifford_algebras}
\bibliographystyle{amsplain}

\end{document}